\documentclass[a4paper, 12pt]{amsart}
\usepackage{hyperref,upref,amssymb,tikz,amssymb, amstext, amscd, amsmath, amsthm}

\numberwithin{equation}{section}

\newtheorem{thm}{Theorem}[section]

\newtheorem{prop}[thm]{Proposition}

\theoremstyle{definition}
\newtheorem{defn}[thm]{Definition}
\newtheorem{ntn}[thm]{Notation}
\theoremstyle{remark}

\newtheorem{rmk}[thm]{Remark}

\newcommand{\upchi}{{\raise.35ex\hbox{\ensuremath{\chi}}}}

\newcommand{\id}{{\operatorname{id}}}

\title[The Primitive Ideal Space of $C(X) \rtimes \mathbb{N}$]{The Primitive Ideal Space of $C(X) \rtimes \mathbb{N}$}

\date{10, Apr, 2025}
\author[X. Chen]{Xiaohui Chen}
\address{Xiaohui Chen,
Department of Mathematics and Physics, North China Electric Power University, Beijing 102206, China}
\email{xiaohui20720@126.com}
\author[H. Li]{Hui Li}
\address{Hui Li,
Department of Mathematics and Physics, North China Electric Power University, Beijing 102206, China}
\email{lihui8605@hotmail.com}
\email{50902471@ncepu.edu.cn}

\thanks{The second author is the corresponding author.}
\begin{document}

\begin{abstract}
We describe the primitive ideal spaces and the Jacobson topologies of a special class of topological graph algebras.
\end{abstract}

\maketitle

\section{Introduction}

The primitive ideal spaces of topological graph algebras and their Jacobson topologies are very hard to characterize. Katsura in his preprint \cite{Kat21} lists all prime ideals of the C*-algebra arising from any singly generated dynamical system (Katsura did not assume the second countability in \cite{Kat21}, see also \cite{Kat06}). Brix, Calrsen, and Sims in \cite{BCS23} characterized the Jacobson topology of the primitive ideal space of the groupoid C*-algebra arising from any commuting local homeomorphisms with harmonious families of bisections. Very recently, Christensen and Neshveyev in \cite{CN241} described the Jacobson topology of the primitive ideal space of the C*-algebra of any etale groupoid with abelian isotropy. Both the work of Brix-Calrsen-Sims and the work of Christensen-Neshveyev cover the case of topological graph algebras. Furthermore Christensen and Neshveyev in \cite{CN242} gave a description of the primitive ideal space of the C*-algebra of any etale groupoid with isotropy groups of local polynomial growth. In this short article, we use semigroup crossed product approach to characterize the primitive ideal spaces and the Jacobson topologies of a special class of topological graph algebras.

\section{Preliminaries}

In this section we recap some C*-algebras background. Firstly, we recall the definition of semigroup crossed products from \cite{LR96} and \cite{Li12}.

\begin{defn}
Let $P$ be a semigroup, let $A$ be a unital C*-algebra, and let $\alpha:P \to \mathrm{End}(A)$ be a semigroup homomorphism such that $\alpha_p$ is injective for all $p \in P$. Define the \emph{semigroup crossed product} $A \rtimes_{\alpha} P$ to be the universal unital C*-algebra generated by the image of a unital homomorphism $i_A:A \to A \rtimes_{\alpha} P$ and a semigroup homomorphism $i_P:P \to \mathrm{Isom}(A \rtimes_{\alpha} P)$ satisfying the following conditions:
\begin{enumerate}
\item $i_P(p)i_A(a)i_P(p)^*=i_A(\alpha_p(a))$, for all $p \in P,a \in A$.
\item For any unital C*-algebra $B$, any unital homomorphism $j_A:A \to B$, any semigroup homomorphism $j_P:P \to \mathrm{Isom}(B)$ satisfying $j_P(p)j_A(a)j_P(p)^*=j_A(\alpha_p(a))$, there exists a unique unital homomorphism $\Phi:A \rtimes_{\alpha} P \to B$, such that $\Phi \circ i_A=j_A$ and $\Phi \circ i_P=j_P$.
\end{enumerate}
\end{defn}

\begin{rmk}
let $A$ be a unital C*-algebra, and let $\alpha$ be an injective endomorphism of $A$. Then $\alpha$ induces a semigroup homomorphism $\alpha:\mathbb{N} \to \mathrm{End}(A)$ such that $\alpha_p=\alpha^p$ for all $p \in P$. The semigroup crossed product $A \rtimes_{\alpha} \mathbb{N}$ can be described as a universal unital C*-algebra generated by the image of a unital homomorphism $i_A:A \to A \rtimes_{\alpha} \mathbb{N}$ and an isometry $s$ of $A \rtimes_{\alpha} \mathbb{N}$ satisfying the following conditions:
\begin{enumerate}
\item $si_A(a)s^*=i_A(\alpha(a))$, for all $a \in A$.
\item For any unital C*-algebra $B$, any unital homomorphism $\rho:A \to B$, any isometry $t$ of $B$ satisfying $t\rho(a)t^*=\rho(\alpha(a))$, there exists a unique unital homomorphism $\Phi:A \rtimes_{\alpha} P \to B$, such that $\Phi \circ i_A=\rho$ and $\Phi(s)=t$.
\end{enumerate}
\end{rmk}

%Pimsner defined Cuntz-Pimsner algebras in \cite{Pim97}. The following definition is a slight generalization of \cite{Pim97}.

%\begin{defn}
%let $X$ be a compact Hausdorff space, let $Y$ be a clopen subset of $X$, let $\sigma:Y \to X$ be a surjective continuous map. Then the Cuntz-Pimsner algebra $\mathcal{O}(\sigma)$ of $\sigma$ is defined to be the universal unital C*-algebra generated by the image of a linear map $j_X:C(Y) \to \mathcal{O}(\sigma)$ and two homomorphisms $j_L:C(X) \to \mathcal{O}(\sigma),j_R:C(Y) \to \mathcal{O}(\sigma)$ such that
%\begin{enumerate}
%\item $j_X((f \circ \sigma) x)=j_L(f)j_X(x)$ for all $x \in C(Y),f \in C(X)$.
%\item $j_X(x)^*j_X(y)=j_R(\overline{x} y)$ for all $x,y \in C(Y)$.
%\item $j_X(\sqrt{f} \circ \sigma)j_X(\sqrt{f} \circ \sigma)^*=j_L(f)$ for all nonnegative $f \in C(X)$.
%\item $j_L(1_{C(X)})=j_R(1_{C(Y)})$.
%\item For any unital C*-algebra $B$, any linear map $\psi:C(X) \to B$, and any homomorphisms $\pi:C(X) \to B, \rho:C(Y) \to B$ satisfying the above four conditions, there exists a unique unital homomorphism $\psi \times \pi \times \rho:\mathcal{O}(\sigma) \to B$ such that $(\psi \times \pi \times \rho) \circ j_X=\psi,(\psi \times \pi \times \rho) \circ j_L=\pi,$ and $(\psi \times \pi \times \rho)\circ j_R=\rho$.
%\end{enumerate}
%\end{defn}

Katsura defined topological graphs and topological graph algebras in \cite{Kat04}.

\begin{defn}
Let $E^0$ and $E^1$ be locally compact Hausdorff spaces, let $r:E^1 \to E^0$ be a continuous map, and let $s:E^1\to E^0$ be a local homeomorphism. Then the quadruple $E=(E^0,E^1,r,s)$ is called a \emph{topological graph}.
\end{defn}

\begin{ntn}
Let $E$ be a topological graph. Denote by $E^\infty:=\{(e_n)_{n=1}^{\infty}\in\prod_{n=1}^{\infty}E^1:s(e_i)=r(e_{i+1}),i\geq 1\}$. An element $(e_n) \in E^\infty$ is said to be \emph{periodic} if there exists $p \geq 2$ such that $(e_1,e_2,e_3,\dots)=(e_p,e_{p+1},e_{p+2},\dots)$. Denote by $E^\infty_{\mathrm{Per}}:=\{(e_n) \in E^\infty:(e_n) \text{ is periodic}\}$. Denote by $E^\infty_{\mathrm{Aper}}:=E^\infty\setminus E^\infty_{\mathrm{Per}}$.
\end{ntn}
%; denote by $E_{-\infty}^{+\infty}:=\{(e_n)_{n=-\infty}^{\infty}\in\prod_{n=-\infty}^{\infty}E^1:s(e_i)=r(e_{i+1}),i\in\mathbb{Z}\}$.

We only consider a special class of topological graphs and their associated C*-algebras. That is, we concentrate on the case that the vertex set and the edge set are the same compact Hausdorff space, the range map is surjective, and the source map is the identity map. More specifically, let $X$ be a compact Hausdorff space and let $\sigma:X \to X$ be a surjective continuous map. Then $E=(E^0,E^1,r,s):=(X,X,\sigma,\id)$ is a topological graph and the definition of this particular kind of topological graph algebras can be reduced to the following form.

\begin{defn}
let $X$ be a compact Hausdorff space and let $\sigma:X \to X$ be a surjective continuous map. Denote by $E=(E^0,E^1,r,s):=(X,X,\sigma,\id)$. Then the topological graph algebra $\mathcal{O}(E)$ of $E$ is defined to be the universal unital C*-algebra generated by the image of a linear map $j_X:C(X) \to \mathcal{O}(E)$ and a homomorphism $j_A:C(X) \to \mathcal{O}(E)$ such that
\begin{enumerate}
\item $j_X((f \circ \sigma) x)=j_A(f)j_X(x)$ for all $x,f \in C(X)$.
\item $j_X(x)^*j_X(y)=j_A(\overline{x} y)$ for all $x,y \in C(X)$.
\item $j_X(\sqrt{f} \circ \sigma)j_X(\sqrt{f} \circ \sigma)^*=j_A(f)$ for all nonnegative $f \in C(X)$.
\item For any unital C*-algebra $B$, any linear map $\psi:C(X) \to B$, and any homomorphism $\pi:C(X) \to B$ satisfying the above three conditions, there exists a unique unital homomorphism $\psi \times \pi:\mathcal{O}(E) \to B$ such that $(\psi\times\pi) \circ j_X=\psi$ and $(\psi\times\pi) \circ j_A=\pi$.
\end{enumerate}
\end{defn}

Finally, we will apply the following version of Williams' theorem to characterize the primitive ideal space of the crossed product by abelian groups (see also \cite{Wil81, Wil07}).

\begin{defn}
Let $X$ be a locally compact Hausdorff space and let $\gamma:X \to X$ be a homomorphism. For $x,y\in X$, define $x \sim y$ if $\overline{\{\gamma^z(x):z \in \mathbb{Z}\}} =\overline{\{\gamma^z(y):z \in \mathbb{Z}\}}$. Then $\sim$ is an equivalent relation on $X$. For $x \in X$, define $[x] :=\overline{\{\gamma^z(x):z \in \mathbb{Z}\}}$, which is called the \emph{quasi-orbit} of $x$. The quotient space $Q(X /\mathbb{Z})$ by the relation $\sim$ is called the \emph{quasi-orbit space}. For $x \in X$, define $\mathbb{Z}_{x}:=\{ z \in \mathbb{Z} : \gamma^z( x) = x\}$ , which is called the \emph{isotropy group} at $x$. For $([x],\lambda ), ([y],\eta ) \in Q(X/ \mathbb{Z})\times \mathbb{T}$, define $([x],\phi )\approx ([y],\psi )$ if $[x] = [y]$ and $\lambda^z=\eta^z$ for all $z \in \mathbb{Z}_{x}$. Then $\approx$ is an equivalent relation on $Q(X /\mathbb{Z}) \times \mathbb{T}$.
\end{defn}

\begin{thm}[{\cite[Theorem~1.1]{LR00}}]\label{Williams thm}
Let $X$ be a locally compact Hausdorff space and let $\gamma:X \to X$ be a homomorphism. Then $\mathrm{Prim}(C_{0}(X)\rtimes_\gamma \mathbb{Z})\cong (Q(X/ \mathbb{Z})\times \mathbb{T})/\approx$.
\end{thm}

\section{Laca's Dilation Theorem Revisited}

Given a semigroup dynamical system $(A,P,\alpha)$, in \cite{CL24} we revisited Laca's dilation theorem when $A$ is a unital commutative C*-algebra and each $\alpha_p$ is injective and unital, in this section we revisit Laca's theorem when $A$ is a unital commutative C*-algebra and each $\alpha_p$ is merely assumed to be injective.

\begin{ntn}
Let $P$ be a subsemigroup of a group $G$ satisfying $G=P^{-1}P$. For $p,q \in P$, define $p \leq q$ if $qp^{-1} \in P$. Then $\leq$ is a reflexive, transitive, and directed relation on $P$.
\end{ntn}

%The following lemma is too elementary to find a reference, but it is very convenient for our construction in the next theorem.

%\begin{lemma}\label{gelfand transform}
%Let $X$ be a compact Hausdorff space and let $\alpha:C(X) \to C(X)$ be an injective (not necessarily unital!) homomorphism. Then there exists a clopen subset $Y$ of $X$ such that $\alpha$ is actually a unital injective homomorphism to $C(Y)$.
%\end{lemma}
%\begin{proof}
%The clopen subset is exactly the support of $\alpha(1_{C(X)})$.
%\end{proof}

\begin{thm}[{cf. \cite[Theorem~2.1]{La00}}]\label{Laca dilation thm}
Let $P$ be a subsemigroup of a group $G$ satisfying $G=P^{-1}P$, let $A=C(X)$ where $X$ is a compact Hausdorff space, and let $\alpha:P \to \mathrm{End}(A)$ be a semigroup homomorphism such that $\alpha_p$ is injective for all $p \in P$. Then there exists a dynamical system $(X_\infty,G,\gamma)$, such that $A \rtimes_{\alpha} P$ is Morita equivalent to $C_0(X_\infty) \rtimes_\gamma G$.
\end{thm}
\begin{proof}
By \cite[Theorem~2.1]{La00}, there exist a C*-dynamical system $(A_\infty,G,\beta)$ such that $A \rtimes_{\alpha}^{e} P$ is Morita equivalent to $A_\infty \rtimes_\beta G$. We cite the proof of \cite[Theorem~2.1]{La00} to sketch the construction of $A_\infty$ and the definition of $\beta$: For $p \in P$, define $A_p:=A$. For $p,q \in P$ with $p \leq q$, define $\alpha_{p,q}:A_p \to A_q$ to be $\alpha_{qp^{-1}}$. Then $\{(A_p,\alpha_{p,q}):p,q \in P,p \leq q\}$ is an inductive system. Denote by $A_\infty:=\lim_{p}(A_p,\alpha_{p,q})$, denote by $\alpha^{p}:A_p \to A_\infty$ the natural unital embedding for all $p \in P$, and denote by $\beta:G \to \mathrm{Aut}(A_\infty)$ the homomorphism satisfying $\beta_{p_0}\circ \alpha^{pp_0}=\alpha^p$ for all $p_0,p\in P$.

For $p \in P$, denote by $X^p_\infty:=\{(x_r)_{r \geq p} \in \prod_{r \geq p}X:ev_{x_s}\circ \alpha_{r,s}=ev_{x_r},\forall s \geq r \geq p\}$. For $(x_r)_{r \geq p} \in X^p_\infty,(y_r)_{r \geq q}\in X^q_\infty$, define $(x_r)_{r \geq p} \sim (y_r)_{r \geq q}$ if $x_r=y_r$ whenever $r\geq p,q$. Define $X_\infty:=(\amalg_{p \in P}X^p_\infty)/\sim$. For $p,q \in P$, define a homeomorphism $\Phi_{p,q}:X^p_\infty \to X^q_\infty,(x_r)_{r \geq p} \mapsto (x_{rq^{-1}p})_{r \geq q}$. For $p_0 \in P$, there exists an injective continuous map from $\amalg_{p \in P}X^p_\infty$ into $\amalg_{p \in P}X^p_\infty$ pieced by $\Phi_{p,pp_0}$, and this injection induces a homeomorphism $\gamma_{p_0}^{-1}:X_\infty \to X_\infty$. Since for $f \in C_0(X_\infty), \beta_{p_0}(f)=f\circ \gamma_{p_0}^{-1}$, we have $C_0(X_\infty) \rtimes_\gamma G \cong A_\infty \rtimes_\beta G$. Hence $A \rtimes_{\alpha} P$ is Morita equivalent to $C_0(X_\infty) \rtimes_\gamma G$ (For the detailed proof of this paragraph, one may refer to the first author's master thesis \cite{Che25}).
\end{proof}

%\begin{ntn}
%We explicitly describe $X_\infty$ and the action of $G$ on $X_\infty$ given in the above theorem.

%\[
%X_\infty=\{(x_p)_{p \in P} \in \prod_{p \in P}X_p:f_{q,p}(x_q)=x_p,\forall p \leq q\}.
%\]

%For $p_0,p,q \in P$ with $q \geq p_0,p$, for $(x_p)_{p \in P} \in X_\infty$, we have

%\[(p_0\cdot (x_p))(p)=x_{pp_0}, (p_0^{-1}\cdot (x_p))(p)=f_{q,p}(x_{qp_0^{-1}}).\]

%Specifically, when $G$ is abelian, we have a simpler form of the group action

% \[\frac{p_0}{q_0} \cdot (x_p)=(f_{q_0}(x_{pp_0})).\]
%\end{ntn}

\section{An Application to Topological Graph Algebras}

Throughout this section, we use the following notation: Let $X$ be a compact Hausdorff space and let $\alpha:C(X) \to C(X)$ be an injective homomorphism. Denote by $\alpha:\mathbb{N} \to \mathrm{End}(C(X)), p \mapsto \alpha^p$. By the Gelfand transform there exist a clopen subset $Y$ of $X$ and a surjective continuous map $\sigma:Y\to X$ such that $\alpha(f)=f\circ \sigma$.

In this situation we are able to simplify the construction of the proof of Theorem~\ref{Laca dilation thm}. For $p \in \mathbb{N}$, we have $X^p_\infty=\{(x_r)_{r \geq p}:(x_p,x_{p+1},x_{p+2},\dots) \in E^\infty\}$. So
\begin{align*}
X_\infty \cong \{(x_r)_{r \geq 0} \in E^\infty:x_0 \in Y\} \amalg \amalg_{p=0}^{\infty}\{(x_r)_{r \geq p} \in E^\infty:x_p \notin Y\}
\end{align*}
and $\gamma_1^{-1}:X_\infty \to X_\infty$ is given by the following formulae
\[
\gamma_1^{-1}(Y\ni x_0,x_1,x_2,\dots)=\sigma(x_0),x_0,x_1,x_2,\dots \in \{(x_r)_{r \geq 0}:(x_r)_{r \geq 0} \in E^\infty\};
\]
\[
\gamma_1^{-1}(Y\not\ni x_p,x_{p+1},x_{p+2},\dots)=x_p,x_{p+1},x_{p+2},\dots \in \{(x_r)_{r \geq p+1} \in E^\infty:x_p \notin Y\}.
\]

Denote a topological graph by $E=(E^0,E^1,r,s):=(X,Y,\sigma,\iota)$. In general, $\mathcal{O}(E)$ is a quotient of $\mathrm{Prim}(C(X) \rtimes_{\alpha} \mathbb{N})$. Under certain condition, they are isomorphic. In this section, we describe the primitive ideal spaces and the Jacobson topologies of $\mathcal{O}(E)$ when $\alpha$ is unital.

\begin{prop}\label{O(E) iso C(X) rtimes N}
Suppose that $\alpha$ is unital. Then $\mathcal{O}(E) \cong C(X) \rtimes_\alpha \mathbb{N}$.
\end{prop}
\begin{proof}
Since $\alpha$ is unital, $Y=X$. Denote by $i_A:C(X) \to C(X) \rtimes_{\alpha} \mathbb{N}$ and by $s$ the unital homomorphism and the isometry generating $C(X) \rtimes_{\alpha} \mathbb{N}$. Denote by $j_X:C(X) \to \mathcal{O}(E)$ and by $j_A:C(X) \to \mathcal{O}(E)$ the linear map and the homomorphism generating $\mathcal{O}(E)$. Define $\rho:=j_A$ and define $t:=j_X(1_{C(X)})^*$. Since $t^*t=j_X(1_{C(X)})j_X(1_{C(X)})^*=j_A(1_{C(X)})=1_{\mathcal{O}(E)}, t$ is an isometry. For any nonnegative $f \in C(X)$, we have
\begin{align*}
t\rho(f)t^*&=j_X(1_{C(X)})^*j_A(f)j_X(1_{C(X)})
\\&=(j_A(\sqrt{f})j_X(1_{C(X)}))^*(j_A(\sqrt{f})j_X(1_{C(X)}))
\\&=j_X(\sqrt{f}\circ \sigma)^*j_X(\sqrt{f}\circ \sigma)
\\&=j_A(f\circ \sigma)
\\&=\rho(\alpha(f)).
\end{align*}
So by the universal property of $C(X) \rtimes_\alpha \mathbb{N}$, there exists a unique homomorphism $\Phi:C(X) \rtimes_\alpha \mathbb{N} \to \mathcal{O}(E)$ such that $\Phi\circ i_A=\rho$ and $\Phi(s)=t$.

On the other hand, define $\psi:C(X) \to C(X) \rtimes_{\alpha} \mathbb{N}$ by $\psi(x):=s^*i_A(x)$, which is a linear map. Define $\pi:=i_A$. For any $f,x \in C(X)$, since $si_A(f)s^*=i_A(f \circ \sigma), i_A(f)s^*=s^*i_A(f\circ \sigma)$. So $\psi(f\circ \sigma x)=s^*i_A(f\circ \sigma x)=s^*i_A(f\circ \sigma )i_A(x)=i_A(f)s^*i_A(x)=\pi(f)\psi(x)$. Observe that $\alpha$ is unital because $\sigma$ is surjective. Then $s$ is actually a unitary (see \cite[Page~5, Remark(2)]{CL24}). So $\psi(x)^*\psi(y)=i_A(x)^*ss^*i_A(y)=i_A(\overline{x}y)=\pi(\overline{x}y)$. For any nonnegative $f\in C(X)$, we have $\psi(\sqrt{f}\circ \sigma)\psi(\sqrt{f}\circ \sigma)^*=s^*i_A(\sqrt{f}\circ \sigma)i_A(\sqrt{f}\circ \sigma)^*s=i_A(f)=\pi(f)$. By the universal property of $\mathcal{O}(E)$, there exists a unique homomorphism $\Psi:\mathcal{O}(E) \to C(X) \rtimes_{\alpha} \mathbb{N}$ such that $\Psi \circ j_X=\psi$ and $\Psi \circ j_A=\pi$.

Finally, for any $x,f \in C(X)$, we compute that
\[
\Phi\circ\Psi(j_X(x))=\Phi(\psi(x))=\Phi(s^*i_A(x))=j_X(1_{C(X)})j_A(x)=j_X(x);
\]
\[
\Phi\circ\Psi(j_A(f))=\Phi(\pi(f))=\Phi(i_A(f))=\rho(f)=j_A(f);
\]
\[
\Psi\circ \Phi(i_A(f))=\Psi(\rho(f))=\Psi(j_A(f))=\pi(f)=i_A(f);
\]
\[
\Psi\circ \Phi(s)=\Psi(j_X(1_{C(X)})^*)=\Psi(j_X(1_{C(X)}))^*=\psi(1_{C(X)})^*=(s^*i_A(1_{C(X)}))^*=s.
\]
So $\Phi\circ\Psi=\id$ and $\Psi\circ \Phi=\id$. Hence $\mathcal{O}(E) \cong C(X) \rtimes_\alpha \mathbb{N}$.
\end{proof}

Suppose that $\alpha$ is unital. By Proposition~\ref{O(E) iso C(X) rtimes N} and Theorem~\ref{Laca dilation thm}, $\mathcal{O}(E)\cong C(X) \rtimes_\alpha \mathbb{N} \cong C_0(X_\infty) \rtimes_\gamma G$. Since $\alpha$ is unital, $Y=X$. So $X_\infty=E^\infty,\gamma_{1}^{-1}(e_1,e_2,e_3,\dots)=\sigma(e_1),e_1,e_2,\dots,$ and $\gamma_{1}(e_1,e_2,e_3,\dots)=e_2,e_3,e_4,\dots$. For any $(e_n)\in E^\infty_{\mathrm{Per}}, (e_n)$ is a simple cycle of some length $N \geq 1$ repeating itself. So $[(e_n)]$ consists of finitely many points and $([(e_n)],\lambda)\approx([(e_n)],\eta)$ if and only if $\lambda^N=\eta^N$. For any $(e_n)\in E^\infty_{\mathrm{Aper}}$ and any $\lambda,\eta \in \mathbb{T}$, we have $([(e_n)],\lambda)\approx([(e_n)],\eta)$. By Theorem~\ref{Williams thm}, we have the following result.

\begin{thm}
Suppose that $\alpha$ is unital. Then $\mathrm{Prim}(\mathcal{O}(E))=(\{[(e_n)]:(e_n)\in E^\infty_{\mathrm{Per}}\}\times \mathbb{T}/\approx)\amalg \{[(e_n)]:(e_n)\in E^\infty_{\mathrm{Aper}}\}$.%, and the Jacobson topology of $\mathrm{Prim}(\mathcal{O}(E))$ is the quotient topology.
\end{thm}

\subsection*{Acknowledgments}

The first author wants to thank the second author for his encouragement and patient supervision.


\begin{thebibliography}{15}
%\bibitem{Bla06} B. Blackadar, Operator algebras, Theory of $C\sp *$-algebras and von Neumann algebras, Operator Algebras and Non-commutative Geometry, III, Springer-Verlag, Berlin, 2006, xx+517.
\bibitem{BCS23} K.A. Brix, T.M. Carlsen, and A. Sims, \emph{Ideal structure of C*-algebras of commuting local homeomorphisms}, preprint, arXiv:2303.02313.
%\bibitem{CS17} T.M. Carlsen and A. Sims, \emph{On {H}ong and {S}zyma\'{n}ski's description of the primitive-ideal space of a graph algebra}, Abel Symp., 12, Operator algebras and applications---the {A}bel {S}ymposium 2015, 115--132, Springer, [Cham], 2017.
\bibitem{Che25} X. Chen, Primitive ideal space of semigroup crossed products arising from rings (Chinese), Master Thesis--North China Electric Power University, 2025.
\bibitem{CL24} X. Chen and H. Li, \emph{Primitive ideal space of $C^*(R_+)\rtimes R^\times$}, preprint, arXiv:2408.09863.
%\bibitem{HS04} J.H. Hong and W. Szyma{\'n}ski, \emph{The primitive ideal space of the {$C\sp \ast$}-algebras of infinite graphs}, J. Math. Soc. Japan \textbf{56} (2004), 45--64.
\bibitem{CN241} J. Christensen and S. Neshveyev, \emph{The primitive spectrum of C*-algebras of etale groupoids with abelian isotropy}, preprint, arXiv:2405.02025.
\bibitem{CN242} J. Christensen and S. Neshveyev, \emph{The ideal structure of C*-algebras of etale groupoids with isotropy groups of local polynomial growth}, preprint, arXiv:2412.11805.
\bibitem{Kat04} T. Katsura, \emph{A class of {$C^\ast$}-algebras generalizing both graph algebras and homeomorphism {$C^\ast$}-algebras {I}. {F}undamental results}, Trans. Amer. Math. Soc. \textbf{356} (2004), 4287--4322.
\bibitem{Kat06} T. Katsura, \emph{A class of {$C\sp *$}-algebras generalizing both graph algebras and homeomorphism {$C\sp *$}-algebras {III}. {I}deal structures}, Ergodic Theory Dynam. Systems \textbf{26} (2006), 1805--1854.
\bibitem{Kat21} T. Katsura, \emph{Ideal structure of C*-algebras of singly generated dynamical systems}, preprint, arXiv:2107.10422.
\bibitem{La00} M. Laca, \emph{From endomorphisms to automorphisms and back: dilations and full corners}, J. London Math. Soc. \textbf{61} (2000), 893--904.
\bibitem{LR96} M. Laca and I. Raeburn, \emph{Semigroup crossed products and the Toeplitz algebras of nonabelian groups}, J. Funct. Anal. \textbf{139} (1996), 415--440.
\bibitem{LR00} M. Laca and I. Raeburn, \emph{The ideal structure of the Hecke C*-algebra of Bost and Connes}, Math. Ann. \textbf{318} 2000, 433--451.
\bibitem{Li12} X. Li, \emph{Semigroup C*-algebras and amenability of semigroups}, J. Funct. Anal. \textbf{262} (2012), 4302--4340.
%\bibitem{Pim97} M.V. Pimsner, \emph{A class of {$C^*$}-algebras generalizing both {C}untz-{K}rieger algebras and crossed products by {${\bf Z}$}}, Fields Inst. Commun., 12, Free probability theory ({W}aterloo, {ON}, 1995), 189--212, Amer. Math. Soc., Providence, RI, 1997.
%\bibitem{Ren00} J. Renault, \emph{Cuntz-like algebras}, Operator theoretical methods (Timi\c soara, 1998), 371--386, Theta Found., Bucharest, 2000.
\bibitem{Wil81} D.P. Williams, \emph{The topology on the primitive ideal space of transformation group C*-algebras and CCR transformation group C*-algebras}, Trans. Amer. Math. Soc. \textbf{266} (1981), 335--359.
\bibitem{Wil07} D.P. Williams, Crossed products of C*-algebras, American Mathematical Society, 2007, xvi+528.
\end{thebibliography}
\end{document}